\documentclass[12pt]{article}
\usepackage{amssymb,amsmath,amsthm, commath, mathrsfs,  secdot, enumitem}
\usepackage[T1]{fontenc}
\usepackage{bbm}

\usepackage{xcolor}
\usepackage{authblk}
\DeclareMathOperator{\I}{\mathbbm{1}}%
\DeclareMathOperator{\Law}{Law}%

\def\E{\hskip.15ex\mathsf{E}\hskip.10ex}
\def\P{\mathsf{P}}

\def\eps{\varepsilon}
\def\phi{\varphi}

\newtheoremstyle{Assump}%
  {3pt}
  {3pt}
  {\itshape}
  {}
  {\bfseries}
  {.}
  {.5em}
  {\thmname{#1} \thmnumber{#2} \thmnote{\normalfont#3}}

\newtheorem{Theorem} {Theorem}[section]
\newtheorem{Lemma}[Theorem]{Lemma}
\newtheorem{Proposition}[Theorem]{Proposition}

\theoremstyle{definition}\newtheorem{Example}[Theorem]{Example}
\theoremstyle{definition}\newtheorem{Remark}[Theorem]{Remark}
\theoremstyle{definition}\newtheorem{Definition}[Theorem]{Definition}

\newtheorem{innercustomthm}{Assumption}
\newenvironment{Assumption}[1]
  {\renewcommand\theinnercustomthm{#1}\innercustomthm}
  {\endinnercustomthm}

\numberwithin{equation}{section}

\renewcommand{\ge}{\geqslant}
\renewcommand{\le}{\leqslant}

\newcommand{\wt}{\widetilde}

\newcommand{\N}{\mathbb{N}}
\newcommand{\R}{\mathbb{R}}

\newcommand{\Z}{\mathbb{Z}}

\newcommand{\B}{\mathcal{B}}

\newcommand{\cP}{\mathcal{P}}
\newcommand*\diff{\mathop{}\!\mathrm{d}}

\newcommand{\limes}{\lim\limits_{n \to \infty}}
\newcommand{\limessup}{\limsup\limits_{n \to \infty} }

\newcommand{\union}[2]{\overset{#2}{\bigcup\limits_{#1}}\;}

\newcommand{\be}{\begin{equation}}
\newcommand{\ee}{\end{equation}}
\newcommand{\ba}{\begin{aligned}}
\newcommand{\ea}{\end{aligned}}

\title{Asymptotic strong Feller property and local weak irreducibility via generalized couplings}

\author[1,$*$]{Oleg Butkovsky}
\author[2,$**$]{Fabrice Wunderlich}

\affil[1]{\small{Weierstrass Institute,

Mohrenstrasse 39, 10117 Berlin, Germany.\bigskip
}}

\affil[2]{\small {Technische Universit\"at Berlin,

Institut f\"ur Mathematik, MA 7-5, Fakult\"at II,

Strasse des 17.~Juni 136, 10623 Berlin, Germany. \bigskip
}}

\begin{document}
\maketitle
\renewcommand{\thefootnote}{*}
\footnotetext{Email: \texttt{oleg.butkovskiy@gmail.com}. Supported in part by DFG Research Unit FOR 2402 and ERC grant 683164.}

\renewcommand{\thefootnote}{**}
\footnotetext{Email: \texttt{fabrice.wunderlich@posteo.de}.   }

\begin{abstract}
In this short note we show how the asymptotic strong Feller property (ASF) and local weak irreducibility can be established via generalized couplings. We also prove that a stronger form of ASF together with local weak irreducibility implies uniqueness of an invariant measure. The latter result is optimal in a certain sense and complements some of the corresponding results of Hairer, Mattingly (2008).
\end{abstract}

\section{Introduction}
In this short note we show how the asymptotic strong Feller property (ASF) and local weak irreducibility can be established via generalized couplings. We also prove that a stronger form of ASF together with local weak irreducibility implies uniqueness of an invariant measure. The latter result is optimal in a certain sense and complements some of the corresponding results of \cite[Section~2]{HM11}, \cite[Section~2.1]{HM08}, and extends some of the ideas of \cite[Section 2]{BKS}.

A central question in the theory of Markov processes can be posed as follows: given a Markov semigroup determine whether it has a unique invariant measure. If the semigroup is strong Feller (this is typically the case for finite--dimensional Markov processes), then this problem is relatively easy to solve. For example, one way to get unique ergodicity for a strong Feller semigroup is just to verify a certain
accessibility condition.

The problem becomes much more difficult if the Markov semigroup is only Feller and not strong Feller (this might happen for infinite--dimensional Markov processes, e.g. stochastic delay equations or SPDEs). A breakthrough was achieved in a series of works by Hairer and Mattingly  \cite{HM}, \cite{HM08}, \cite{HM11}, where the notion of an \textit{asymptotically strong Feller} (ASF) Markov process was introduced. It turned out that if a Markov process has the ASF property, then any two of its ergodic invariant measures have disjoint support. This, in turn, implies that ASF together with a certain irreducibility condition yields unique ergodicity. However, verification of the ASF condition in practice might be quite involving and is usually based on Malliavin calculus techniques.

This short article has two goals. First, we provide an alternative way of establishing the ASF property based on the generalized couplings technique. We hope also that it might be useful in obtaining certain gradient-type bounds for SPDEs. Second, we show that a stronger version of the ASF property together with local weak irreducibility (a weaker condition than the one used in \cite{HM}) is sufficient for unique ergodicity. We also provide a way how this local weak irreducibility can be established via generalized couplings.
\medskip

\noindent \textbf{Acknowledgments}. This article is  based on the master  thesis of FW written at TU~Berlin under the supervision of OB. The authors are grateful to Michael Scheutzow for useful discussions. OB has received funding from the European Research Council
(ERC) under the European Union's Horizon 2020 research and innovation program (grant agreement No.~683164).

\section{Main results}

First, let us introduce some basic notation. Let $(E, d)$ be a Polish space equipped with the Borel $\sigma$-field $\mathcal{E}=\B(E)$. For $x\in E$, $R>0$ let $B_R(x):=\{z \in E\colon d(z,x)< R\}$ be the open ball of radius $R$ around $x$.  Denote by $\cP(E)$ the set of all probability measures on $(E, \mathcal{E})$. For $\mu, \nu \in \cP(E)$ let $\mathscr{C}(\mu, \nu)$ be the set of all \textit{couplings} between $\mu$ and $\nu$, i.e. probability measures on $(E \times E, \mathcal{E} \otimes \mathcal{E})$ with marginals $\mu$ and $\nu$. For $\mu, \nu \in \cP(E)$ and a measurable function $\rho: E\times E\to \R_+$ we put
\begin{equation}
W_\rho(\mu, \nu) \; := \; \inf_{\gamma\, \in\, \mathscr{C}(\mu, \nu)} \;  \int_{E \times E} \,  \rho(x,y)\; \gamma(\diff x, \diff y), \qquad \mu, \nu \in \cP(E). \label{2.1}
\end{equation}
Clearly, for $\rho=d$, the function $W_d$ is just the standard Wasserstein--$1$ (or Kantorovich) distance. If $\rho(x,y)=\I(x=y)$, then $W_\rho$ coincides with the \textit{total variation} distance $d_{TV}$. The latter can also be defined as follows
$$
d_{TV}(\mu, \nu)  \; :=  \; \sup_{A \, \in \, \mathcal{E}}  \;  |\mu(A) - \nu(A)|.
$$

A mapping $\rho: E \times E \to \R_+$ is called a \textit{pseudo-metric}, if it satisfies all characteristics of a metric without possibly the property that $\rho(x,y)=0$ implies $x=y$.
For a function $f\colon D\to\R$, where $D$ is an arbitrary domain, we will denote $\|f\|_{\infty}:=\sup_{x\in D}|f(x)|$.

For the convenience of the reader, all of our results will be stated in the continuous time  framework. Yet, they also hold for the discrete time setup. Consider a Markov transition function $\{P_t(x, A), x\in E, A\in\mathcal{E}\}_{t\in\R_+}$.  Recall the following concepts introduced in \cite{HM}.

\begin{Definition}[{\cite[Definition 3.1]{HM}}] An increasing sequence $(d_n)_{n\in\Z_+}$ of bounded, continuous pseudo-metrics on $E$ is called \textit{totally separating} if for every $x,y\in E$, $x\neq y$ one has $\limes d_n(x,y) = 1$.
\end{Definition}

\begin{Definition}[{\cite[Definition 3.8]{HM}}] We say that a Markov semigroup $(P_t)_{t\in \R_+}$ satisfies the \textit{Asymptotic Strong Feller Property} (ASF) if it is Feller and  for every $x \in E$ there exist a sequence of positive real numbers $(t_n)_{n\in\Z_+}$ and a totally separating sequence $(d_n)_{n\in\Z_+}$ of pseudo--metrics such that
$$
\inf_{U\, \in \,  \mathcal{U}_x} \; \limessup \; \sup_{y\, \in \, U}  \; W_{d_n}\left( P_{t_n}(x,\cdot) , P_{t_n}(y, \cdot)\right) \; = \;  0,
$$
where $\mathcal{U}_x:=\{ U \subseteq E\colon \, x \in U \text{ and } U \text{ open}\}$ denotes the collection of all open neighborhoods of $x$.
\end{Definition}

We refer to  \cite{HM}, \cite{HM08}, \cite{HM11} for further discussions of this notion. In particular, it was shown in \cite[Corollary 3.17]{HM} that ASF, together with a certain  irreducibility assumption, implies uniqueness of invariant probability measure. It was also shown there (\cite[Proposition 3.12]{HM}) that ASF follows from the following stronger property.

\begin{Definition}[see also {\cite[Proposition 3.12]{HM}}] \label{Def:ASF+} We say that a Markov semigroup $(P_t)_{t\in \R_+}$ satisfies the \textit{Asymptotic Strong Feller Plus Property} (ASF+) if it is Feller and there exist $x_0 \in E$, a non-decreasing sequence $(t_n)_{n\in\Z_+}$, a positive sequence $(\delta_n)_{n\in\Z_+}$ with $\delta_n \searrow 0$ as $n \to \infty$ and a non-decreasing function $F: [0,\infty) \to [0,\infty)$ such that for every $d$-Lipschitz continuous function $\varphi: E \to \R$ with Lipschitz constant $K>0$, we have
\begin{equation} \label{eq:ASF+}
\left| P_{t_n} \phi (x)  - P_{t_n} \phi (y) \right| \;   \le  \; d(x,y)  \;  F(d(x,x_0) \vee d(y,x_0) ) \; ( \, \norm{\phi}_\infty +  \delta_n K)
\end{equation}
for every $n \in \N$ and $x,y \in E$.
\end{Definition}

\begin{Remark} \label{rem:2.4}
It was shown in \cite[Remark 3.10]{HM} that ASF is an extension of the strong Feller property (i.e. any strong Feller semigroup satisfies ASF). Furthermore, \cite[Proposition 3.12]{HM} proved that ASF follows from ASF+. However, according to Example~\ref{ex:2} below, ASF+ is \textbf{not} an extension of the strong Feller property, and one can construct a strong Feller semigroup that does not satisfy ASF+. Nevertheless, the ASF+ property is quite useful, since it holds for many interesting Feller but not strong Feller semigroups.
\end{Remark}

As mentioned before, thanks to the results of \cite{HM}, to show unique ergodicity it is enough to establish ASF or ASF+ (together with some irreducibility-type conditions). However verifying this criteria in practice is usually  rather tedious (see, e.g., \cite[Proposition 4.2]{CGHV}) and involves Malliavin calculus techniques. Our first main result suggests a different strategy of verifying ASF+. This strategy is based on the generalized coupling method and develops some ideas of \cite[Section 2.2]{BKS}.

Consider the following assumption which is related to \cite[Assumption A]{BKS}.

\begin{Assumption}{\textbf{A1}}\label{A:1}
There exist non-decreasing functions $F_1, F_2\colon \R_+ \to \R_+$, a non-increasing function
$r\colon\R_+\to\R_+$ with $\lim_{t\to\infty} r(t)=0$ and $x_0 \in E$ such that for every $x,y \in E$ and $t \in \R_+$, there exist $E$-valued random variables $Y^{x,y}_t$ and $Z^{x,y}_t$ on a common probability space with the following properties
\begin{enumerate}
\item $\Law(Y^{x,y}_t)=P_t(y,\cdot)$ and
$$
d_{TV} \left(\Law(Z^{x,y}_t), P_t(x,\cdot)\right) \; \le \; F_1(d(x,x_0) \vee d(y,x_0))\;  d(x,y), \quad t\ge0.
$$
\item $\E d(Z^{x,y}_t, Y^{x,y}_t) \;  \le  \; F_2(d(x,x_0) \vee d(y,x_0)) \; r(t) \; d(x,y), \quad t\ge0$.
\end{enumerate}
\end{Assumption}

\begin{Theorem}\label{Thm:2.2}
If $(P_t)_{t\in \R_+}$ satisfies Assumption \ref{A:1}, then it also satisfies ASF+ with $F:=2F_1+F_2$ (and hence $(P_t)_{t\in  \R_+}$ is asymptotically strong Feller).
\end{Theorem}

It is interesting to compare \cite[Assumption A]{BKS} and Assumption~\ref{A:1}. Both of them are similar generalized-coupling-type assumptions which yield  certain mixing properties. However, the former is \textbf{global}  in space, whilst the latter is \textbf{local} in space; this difference is crucial for studying ergodic properties of certain SPDE models, see \cite[Section 5]{BKS}. Therefore we believe that Assumption~\ref{A:1} is more suited for establishing exponential ergodicity than \cite[Assumption A]{BKS}.

As a possible application of this result let us mention that it was shown in \cite{BKS} that the fractionally dissipative Euler model admits a generalized coupling satisfying Assumption~\ref{A:1}. Thus, by Theorem~\ref{Thm:2.2}, the gradient-type bound from \cite[Proposition 4.2]{CGHV} holds.

Another property, which is important for unique ergodicity is \textit{local weak irreducibility}. The following definition is inspired by \cite[Assumptions 3 and 6]{HM08}.

\begin{Definition} We say that a semigroup $(P_t)_{t\in \R_+}$ is \textit{locally weak irreducible} if there exists $x_0 \in E$ such that for any $R>0$ and $\eps>0$ there exists $T:=T(R,\varepsilon)>0$ such that for any $t\ge T$ one has
\begin{equation}\label{LWI}
\inf_{x,y \, \in \, B_R(x_0)}  \; \; \sup_{\Gamma  \in \, \mathscr{C}(P_t\delta_x, P_t\delta_y)} \;
\Gamma \left(\{(x',y') \in E \times E \; \colon \;  d(x',y')\le \eps \}\right) \; > \;  0.
\end{equation}
\end{Definition}

Our second main result provides a sufficient condition for local weak irreducibility in terms of generalized couplings. Consider the following assumption, which is the same as \cite[Assumption B2]{BKS}.

\begin{Assumption}{\textbf{A2}}\label{A:2} There exist a set $B\subseteq E$, a function $R:\R_+\to\R_+$ with \mbox{$\lim_{t\to\infty} R(t)=0$}, and $\eps>0$ such that for any $x,y \in B$ and
$t\ge0$, there exist $E$-valued random variables $Y^{x,y}_t$ and $Z^{x,y}_t$ on a common probability space with the following properties
\begin{enumerate}
   \item $\Law(Y^{x,y}_t)=P_t(y,\cdot)$ and
\begin{equation*}
d_{TV}(\Law(Z^{x,y}_t), P_t(x,\cdot)) \; \le \; 1-\eps,\quad t\ge0.
\end{equation*}
\item $\E d(Y^{x,y}_t, Z^{x,y}_t)\; \le\; R(t), \quad t\ge0$.
 \end{enumerate}
\end{Assumption}

\begin{Theorem}\label{Thm:2.3}
If there exists $x_0 \in E$ such that for all $M>0$ the semigroup $(P_t)_{t\in \R_+}$ satisfies Assumption~\ref{A:2} for the set $B:=B_M(x_0)$ (with some $\eps=\eps(M)>0$), then $(P_t)_{t\in \R_+}$ is locally weak irreducible.
\end{Theorem}

Finally, the following theorem illustrates the use of these notions.

\begin{Theorem}\label{Thm:2.4}
If $(P_t)_{t\in \R_+}$ is locally weak irreducible and satisfies ASF+ with a non-decreasing function $F:\R_+ \to [0,\infty)$ such that $\|F\|_\infty <\infty$, then $(P_t)_{t\in \R_+}$ possesses at most one invariant probability measure.
\end{Theorem}

Recall that it was shown in \cite[Theorem 2.5]{HM08} that \textbf{global} weak irreducibility and ASF+ with $\norm{F}_\infty <\infty$ additionally imply existence of an invariant measure and exponential ergodicity. Theorem~\ref{Thm:2.4} shows that if the semigroup satisfies \textbf{local} rather than \textbf{global} weak irreducibility, then uniqueness of invariant measure is guaranteed. This result is optimal in the following sense: the given assumptions do not guarantee existence of an invariant probability measure (see Example \ref{Ex:2}); furthermore the requirement  $\norm{F}_\infty <\infty$ cannot be dropped (see Example \ref{Ex:1}).

In addition, we note that Theorem~\ref{Thm:2.4} complements  \cite[Corollary 3.17]{HM}. The latter shows unique ergodicity provided that the semigroup satisfies
a stronger condition than local weak irreducibility and a weaker condition than ASF+ with finite \,$\norm{F}_\infty$.

\begin{Example} \label{Ex:1} This example shows that local weak irreducibility together with ASF+ lacking the requirement  $\norm{F}_\infty <\infty$  does not guarantee unique ergodicity (therefore it implies that ASF together with local weak irreducibility are insufficient for unique ergodicity).

Fix $\xi \in (0,2^{-1})$ and consider the state space $E:=\N \cup \{n+ \frac{\xi}{n}\mid n \in \N\} $ equipped with the standard Euclidean distance. Consider the Markov transition function $(P_t)_{t \in \Z_+}$ induced by
\begin{align*}
&P_1(x,A)=
  \begin{cases} \frac{1}{2} \delta_{1}(A) +  \frac{1}{2} \delta_{x+1}(A),& \quad \text{ if } 								x \in \N \\[1ex]
								\frac{1}{2} \delta_{1+ \xi }(A)  +  \frac{1}{2} \delta_{\lfloor x \rfloor+1+								\frac{\xi}{\lfloor x \rfloor+1}}(A),& \quad \text{ if } x \in \{n+
								\frac{\xi}{n} \mid n \in \N\},
  \end{cases}
\end{align*}
where $A \in 2^E= \B(E)$. In other words, from any positive integer the process goes to the next integer with probability $\frac{1}{2}$ and to $1$ with probability $\frac{1}{2}$. Similarly, from any shifted positive integer, it moves to the next  shifted integer with probability $\frac{1}{2}$ and to $1+ \xi$ with probability $\frac{1}{2}$.

Clearly, this dynamic has two invariant probability measures: one sits on the integers and the other sits on the shifted integers:
$$
\mu_1 :=  \sum_{i=1}^\infty \frac{1}{2^i} \delta_i \qquad \text{ and } \qquad \mu_2 :=  \sum_{i=1}^\infty \frac{1}{2^i}  \delta_{i+ \frac{\xi}{i}}
$$

To see that the semigroup has the local weak irreducibility property \eqref{LWI}, fix arbitrary $\eps>0$, $x,y \in E$. Choose $K \in \N$ such that $\xi/K< \varepsilon$ and define $A_K:= \{K, K+\xi/K\}$. Then, obviously $A_K\times A_K  \subseteq \{(x',y') \in E \times E\colon |x'-y'| \le \varepsilon\}$ and we have
\begin{align*}
\sup_{\Gamma \, \in \, \mathscr{C}(P_{K+1}\delta_x, P_{K+1}\delta_y)} \Gamma \left(\{(x',y') \in E \times E\, \colon \,d(x',y')\le \eps \}\right)  \; &\ge \; \sup_{\Gamma \, \in \, \mathscr{C}(P_{K+1}\delta_x, P_{K+1}\delta_y)} \;  \Gamma \left(A_K \times A_K \right) \\
&\ge  \; P_{K+1}(x, A_K) \; P_{K+1}(y, A_K) \\
&\ge \; \frac{1}{2^{2K+2}}  \; > \;  0.
\end{align*}
Hence, bound \eqref{LWI} holds for arbitrary $x,y \in E$, and thus this semigroup is locally weak irreducible.

Finally, let us show now that $(P_t)_{t\in \Z_+}$ satisfies ASF+. Obviously the semigroup is strong Feller, however by Remark \ref{rem:2.4} this only implies that $(P_t)_{t\in \Z_+}$ is ASF rather than ASF+; therefore we have to check the ASF+ property directly. Choose $t_n:=1$ and $\delta_n:=0$, $n\in\N$ and $x_0:=1$. Take any function $\phi: E \to \R$ which is Lipschitz with constant $K$.
Then, clearly,
\begin{equation}
\label{step1asfex}
|P_1\phi(x) - P_1\phi(y) |  \; \le \;  2\norm{\phi}_{\infty}, \quad x,y\in E
\end{equation}
It is easy to see, that if $x\neq y$, then
\begin{equation*}
|x-y|  \; \ge \; \frac{\xi}{x\vee y} \; = \;  \frac{\xi}{((x-1)\vee (y-1))+1}.
\end{equation*}
Combining this with \eqref{step1asfex}, we get
$$
| P_1\phi(x) - P_1\phi(y) |  \; \le \; 2\norm{\phi}_{\infty}  \;  \le \;  2 |x-y| \; \frac{(|x-1|\vee |y-1|)+1}{\xi} \; \norm{\phi}_{\infty}.
$$
Thus, inequality \eqref{eq:ASF+} holds with $F(u):=\frac{2u+1}{\xi}$, $u\ge0$. Therefore, this process is locally weak irreducible and satisfies ASF+ (and is even strong Feller), yet it has two invariant probability measures.
\end{Example}

\begin{Example} \label{Ex:2}
This example shows that the assumptions of Theorem~\ref{Thm:2.4} do not guarantee existence of an invariant probability measure.

Let $E=\R$ and consider a Markov semigroup corresponding to the standard Brownian motion $(W_t)_{t \in \R_+}$ on some probability space $(\Omega, \mathcal{F}, \P)$, i.e. $P_t(x,\cdot)=\operatorname{Law}(x+W_t)$ for any $x \in \R$, $t\ge 0$. This semigroup has the local weak irreducibility property as well as ASF+ with bounded $F$.

To establish ASF+ with bounded $F$, set $t_n=1$, $\delta_n=0$, $n \in \N$ and without loss of generality let $\phi: \R \to \R$ be bounded. Let $p_t$ be the density of a Gaussian random variable with mean $0$ and variance $t>0$. Then, applying \cite[Bound (2.4e)]{Ros87}, we deduce
that there exist positive constants $M,\alpha$ such that for any $x,y\in\R$ one has
\begin{align*}
|P_1\phi(x) - P_1\phi(y) |& \;  \le  \;  \norm{\phi}_{\infty} \; \int_\R\, \left| p_1(z-x) -p_1(z-y) \right|\; \diff z\\
& \; \le \; M \; |x-y|\; \norm{\phi}_{\infty}  \; \int_\R  \, (e^{- \alpha(z-x)^2} + e^{- \alpha (z-y)^2})\;\diff z\\
& \; = \; 2 M \;  \sqrt{\frac{\alpha}{\pi}}\;  |x-y| \; \norm{\phi}_{\infty}.
\end{align*}
Hence, the semigroup satisfies ASF+ with the constant $F:= 2M \sqrt{\frac{\alpha}{\pi}}$.

Now let us show local weak irreducibility. Let $x_0=0$, $T=1$. Choose any $R>0$, $\eps>0$, $t\ge1$. Let $\Gamma^{x,y}:= P_t\delta_x \otimes P_t\delta_y$ be the independent coupling of $P_t\delta_x$ and $P_t\delta_y$ for any $x,y \in B_R(0)$. Then,
\begin{align*}
\Gamma^{x,y}\left( \left\{ x',y' \in \R\colon |x'-y'| \; \le \; \varepsilon \right\}\right) \; &\ge \;
 P_t\left( x , \left[ -\eps/2, \eps/2 \right] \right)  P_t\left( y , \left[ -\eps/2, \eps/2 \right] \right) \\
& = \;  \int_{-\eps/2}^{\eps/2}  p_t(x-z)\, d z \;  \int_{-\eps/2}^{\eps/2}  p_t(y-z)\, \diff z
\end{align*}
for every $x,y \in B_{R}(0)$. However, for every $x \in B_{R}(0)$,  $z \in [-\eps/2, \eps/2]$, we obviously have
$$
p_t(x-z) \;  \ge  \; \frac{1}{\sqrt{2\pi t}} \, \exp \Bigl\{ - \frac{( R+\eps/2)^2}{2t} \Bigr\}  =: \lambda \;   > \; 0
$$
and thus,
\begin{equation*}
\Gamma^{x,y}\left( \left\{ x',y' \in \R\colon \; |x'-y'|\le \eps \right\}\right) \; \ge \;
\lambda^2 \eps^2  \; > \;  0
\end{equation*}
for every $x,y \in B_{R}(0)$, which implies local weak irreducibility. Hence, the assumptions of Theorem \ref{Thm:2.4} are satisfied. On the other hand, as it is well-known, the semigroup $(P_t)_{t \in \R_+}$ has no invariant probability measure.
\end{Example}

Finally, let us present an example showing that the strong Feller property does not imply ASF+.

\begin{Example}\label{ex:2}
Let $E:=[0,3]$ be endowed with the Euclidean distance and let $\zeta$ be a random variable uniformly distributed on $[0,1/3]$. Define a Markov transition function $(P_n)_{n\in\Z_+}$ on $([0,3], \B([0,3]))$ by
$$
P_1(x, \cdot) := \begin{cases} \Law\left( 2- \sqrt{x} + \zeta \right), \qquad \quad x \in [0,1]; \\ \Law\left( \frac{2}{3} + \frac{x}{3} + \zeta \right),  \qquad \quad \; \;  x \in [1,3]. \end{cases}
$$

We claim that this semigroup is stong Feller but ASF+ property does not hold for this semigroup.

We start by showing the strong Feller property. Let $\phi: [0,3] \to \R$ be a bounded measurable function. Our goal it to show that the function $P_1\phi$ is continuous. Therefore, let $x \in [0,3]$ and $(x_n)_{n\in\Z_+} \subseteq [0,3]$ be a sequence converging to $x$. We need to show
\begin{equation}\label{contbound}
| P_1\phi(x) - P_1\phi(x_n) | \; \to \; 0, \quad\text{as $n\to\infty$}.
\end{equation}

Note that for any $y_1, y_2 \in [0,1]$ with $y_1 \le y_2$, we have
\begin{align*}
d_{TV}( P_1 \delta_{y_1} , P_1 \delta_{y_2} ) \; &= \; \frac32 \int_{\R} \, \bigl| \I_{[2-\sqrt{y_1},  2-\sqrt{y_1}+\frac{1}{3}]}(z)   -  \I_{[2-\sqrt{y_2} ,  2-\sqrt{y_2}+\frac{1}{3}]}(z)  \bigr|\;\diff z\\[1ex]
&\le \; 3(\sqrt{y_2}  - \sqrt{y_1}).
\end{align*}
Similarly, in the case where $y_1, y_2 \in [1,3]$, one gets $d_{TV}( P_1 \delta_{y_1} , P_1 \delta_{y_2}) \le y_2-y_1$. This implies, that
\begin{equation*}
\lim_{n\to\infty} \; d_{TV}(P_1(x,\cdot),P_1(x_n,\cdot))\;=\;0,
\end{equation*}
which, in turn, yields \eqref{contbound} and establishes the strong Feller property for this semigroup.

Now let us show that ASF+ fails for this process. Assume for the sake of a contradiction that ASF+ holds, i.e., that there exist a non-decreasing sequence
$(t_n)_{n\in\Z_+} \subseteq \Z_+$, a sequence $(\delta_n)_{n\in \Z_+} \subseteq \R_+$ with $\delta_n \searrow 0$ as $n\to \infty$, and a non-decreasing function $F:[0, \infty) \to [0, \infty)$ such that for every Lipschitz continuous
$\phi: [0,3] \to \R$ with Lipschitz constant $K>0$ it holds
$$
\left| P_{t_n} \phi (x) - P_{t_n} \phi (y) \right|  \;  \le \;  |x-y|  \; F(|x| \vee |y|) \; (\, \norm{\varphi}_\infty + \delta_n K),\quad x,y \in [0,3],\, n\in \N.
$$
Taking in the above inequality $n=1$, $x=0$ and $\varphi(z)= z$, we obtain in particular
\begin{equation}\label{eq:11}
\left|P_{t_1}\phi(0) - P_{t_1}\phi(y)\right| \; \le \;  y\, F(1) \, (3 + \delta_1) \quad y \in [0,1].
\end{equation}
By definition of the semigroup $(P_n)_{n\in\N}$, we have for every $n\in\Z_+$
$$
P_n\varphi(u)  =  
\frac{5}{4}(1 - 3^{-n})  + 3^{-n}u,\quad u\in[1,3].
$$
Thus we derive for every $z \in [0,1]$
$$
P_{n} \varphi (z) \; =  \; P_1(P_{n-1} \varphi)(z)  \; = \;
\E [P_{n-1}\varphi(2- \sqrt{z}+ \zeta) ]  \; = \; \frac54(1- 3^{-n+1})  +  \frac{2- \sqrt{z} + 1/6}{3^{n-1}},
$$
where we took into account that for $z \in [0,1]$ we have $2-\sqrt{z} + \zeta \in (1,3]$ almost surely.

Therefore, we finally get
$$
\left| P_{n} \varphi (0)  -  P_{n} \varphi (y) \right|  \; = \;  3^{-n+1}\sqrt{y}.
$$
for every $y\in [0,1]$ and $n\in \N$. Combining this with \eqref{eq:11}, we derive that the following should hold for every $y \in [0,1]$:
$$
3^{-t_1+1}\sqrt{y} \; \le \; y \, F(1) \, (3+ \delta_1).
$$
However this is impossible. Therefore this process does not satisfy the ASF+ condition.
\end{Example}

\newpage
\section{Proofs}

\begin{proof}[Proof of Theorem \ref{Thm:2.2}]
We begin with the following calculation. Let $t\ge0$, $x,y\in E$. Let $\phi\colon E \to \R$ be a
$d$-Lipschitz continuous function with constant $K>0$. Denote $P_t^Z:= \Law(Z^{x,y}_t)$. Then
%
%
%
\begin{align}\label{boundth25}
&|P_{t}\phi(x)  -  P_{t}\phi(y)| \nonumber\\
 &\qquad \le \; \left| \int_E\, \phi(z) \; P_{t}(x,\diff z)  -\int_E\,  \phi(z) \; P_t^Z(\diff z)\right|  \; + \;  \left| \int_E \, \phi(z)\; P_{t}^Z(\diff z)  - \int_E \, \phi(z)  \; P_{t}(y,\diff z)\right| \nonumber\\
&\qquad \le \; 2\norm{\phi}_{\infty} \; d_{TV} \left(\Law(Z^{x,y}_{t}) ,  \Law(X^{x,y}_{t})\right)\; +\;
\int_E  \, |\varphi(Z^{x,y}_{t}) - \varphi(Y^{x,y}_{t})|  \; \diff \P.
\end{align}
Using part 1 of Assumption \ref{A:1}, we bound the first term in the right--hand side
of~\eqref{boundth25}:
\begin{equation}\label{step1bound}
2\norm{\varphi}_\infty \; d_{TV} \left(\Law(Z^{x,y}_{t}) ,  \Law(X^{x,y}_{t})\right) \; \le \;
2F_1(d(x,x_0) \vee  d(y,x_0)) \; \norm{\phi}_\infty \; d(x,y).
\end{equation}
Applying part 2 of Assumption \ref{A:1}, we obtain the following bound on the second term in the right--hand side of \eqref{boundth25}.
\begin{equation*}
\int_E \, |\phi(Z^{x,y}_{t}) - \varphi(Y^{x,y}_{t})|  \; \diff \P \;  \le \;  K \, \E d(Z^{x,y}_{t},Y^{x,y}_{t} ) \; \le \; F_2( d(x,x_0) \vee d(y,x_0))\; r(t) \, K \, d(x,y).
\end{equation*}
Hence, combining this with \eqref{step1bound} and \eqref{boundth25}, we get
\begin{equation}\label{eq:3.1}
| P_{t}\phi(x)  -  P_{t}\phi(y)| \;  \le  \; F_3( d(x,x_0) \vee d(y,x_0)) \;  d(x,y) \;
(  \, \norm{\varphi}_\infty + r(t)K),
\end{equation}
where we denoted $F_3(z):=2F_1(z)+F_2(z)$, $z\ge0$.

Now let us show that $(P_t)_{t\in \R_+}$ satisfies ASF+ property. First, let us show that this semigroup is Feller. Fix $t>0$, $x\in E$. Let $(x_n)_{n\in\Z_+}$ be  a sequence of elements in $E$ converging to $x$. Then, using \eqref{eq:3.1}, we obtain for any bounded $d$-Lipschitz function $\phi\colon E \to \R$ with Lipschitz constant $K>0$
\begin{equation*}
| P_t \phi(x)  -  P_t\phi(x_n) | \; \le \;  F_3( d(x,x_0) \vee d(x_n,x_0)) \; d(x,x_n) \;
(  \, \norm{\varphi}_\infty + r(t)K).
\end{equation*}
Therefore $| P_t \phi(x)  -  P_t\phi(x_n) |\to 0$ as $n\to\infty$. This, together with the
Portmanteau theorem (see, e.g., \cite[Lemma~3.7.2]{Shi16}), implies that $P_t(x_n,\cdot)$ converges weakly to $P_t(x,\cdot)$ as $n\to\infty$. Thus, the semigroup $(P_t)_{t\in\R_+}$ is Feller.

Second, let us verify bound \eqref{eq:ASF+}. For $n\in\N$ put $t_n:=n$, $\delta_n:=r(n)$. We claim know that \eqref{eq:ASF+} holds for the sequences $(t_n)$, $(\delta_n)$, and the function
$F:=F_3$ defined above. Indeed, let $\phi: E \to \R$ be a  $d$-Lipschitz continuous function with constant $K>0$ and $x,y\in E$. Then, applying \eqref{eq:3.1}, we derive for any $n\in\N$
\begin{align*}
| P_{t_n}\phi(x)  -  P_{t_n}\phi(y)|\; &= \; | P_{n}\phi(x)  -  P_{n}\phi(y)|\\[1ex]
&\le \;  F_3( d(x,x_0) \vee d(y,x_0)) \; d(x,y)\; ( \, \norm{\varphi}_\infty + \delta_n K),
\end{align*}
which is \eqref{eq:ASF+}. Therefore all the conditions of Definition~\ref{Def:ASF+} are met and thus the semigroup $(P_t)$ satisfies ASF+ property.
\end{proof}

As a helpful tool in the sequel, we would like to recall the following \textit{Gluing Lemma}.
\begin{Proposition}[{\cite[p. 23-24]{V08}}] \ \label{Gluing Lemma}
Let $\mu_i$, $i=1,2,3$, be probability measures on a Polish space $E$.
If $(X_1, X_2)$ is a coupling of $\mu_1, \mu_2$ and $(Y_2, Y_3)$ is a coupling of $\mu_2, \mu_3$, then there exists a triple of random variables $(V_1, V_2, V_3)$ such that $(V_1,V_2)$ has the same law as $(X_1, X_2)$ and $(V_2, V_3)$ has the same law as $(Y_2, Y_3)$.
\end{Proposition}

\begin{proof}[Proof of Theorem \ref{Thm:2.3}] Fix $M>0$, $\delta>0$ and choose $T\!:=T(\delta)\ge 0$ such that \mbox{$\delta^{-1} R(T) < \frac{\varepsilon}{2}$.} Let $t\ge T$.  Then, using the Markov inequality and part 2 of assumption \ref{A:2}, we get  for any $x,y\in B_M(x_0)$
$$
\P\left( d(Z_t^{x,y} ,Y_t^{x,y}) > \delta  \right)  \; \le \; \frac{\E d(Z_t^{x,y} ,Y_t^{x,y})}{\delta} \; \le \; \frac{R(t)}{\delta} \; \le \; \frac{R(T)}{\delta} \; < \; \frac{\varepsilon}{2},
$$
which is equivalent to
\begin{equation}\label{eq:14}
\P\left(  d(Z_t^{x,y} ,Y_t^{x,y})  \le \delta \right) \;  \ge \; 1- \frac{\varepsilon}{2}.
\end{equation}
By part 1 of Assumption \ref{A:2} and the definition of the total variation distance, there exist random variables $\wt{X}^{x,y}_t, \wt{Z}_t^{x,y}$ such that
$$
d_{TV} \left( P_t(x,\cdot) ,  \Law(Z^{x,y}_t)\right)  \; = \;  \P\left(\wt{X}^{x,y}_t \neq \wt{Z}_t^{x,y} \right) \; \le \; 1- \varepsilon
$$
and therefore
\begin{equation}\label{eq:15}
 \P\left(\wt{X}^{x,y}_t = \wt{Z}_t^{x,y} \right) \; > \;  \varepsilon.
\end{equation}
Now, according to Proposition~\ref{Gluing Lemma}, there exist random variables $V^{X}_t,V^{Z}_t, V^{Y}_t$ on a probability space $(\Omega, \mathcal{F}, \P)$ such that $(V^{X}_t, V^{Z}_t)$ has the same law as $(\wt{X}_t^{x,y}, \wt{Z}_t^{x,y})$ and $(V^{Z}_t, V^Y_t)$ has the same law as $(Z^{x,y}_t, Y_t^{x,y})$. Therefore, using also the fact that for any measurable sets $A, B \in \mathcal{F}$ one has $\P(A\cap B) \ge \P(A)+\P(B)-1$, we deduce
\begin{align*}
\P\left(d(V^X_t ,V^Y_t)  \le   \delta  \right) \; &\ge \;
\P\left( \{V_t^X =V_t^Z\} \cap \{(d(V^Z_t ,V^Y_t) \le  \delta\}\right)\\
&\ge\;  \P(V_t^X =V_t^Z)\; +\; \P(d(V^Z_t ,V^Y_t) \le  \delta)-1\\
&= \; \P(\wt{X}_t^{x,y} =\wt{Z}_t^{x,y})\; + \P(d(Z_t^{x,y} ,Y_t^{x,y}) \le  \delta)-1\\
&\ge\;  \eps/2,
\end{align*}
where the last inequality follows from \eqref{eq:14} and \eqref{eq:15}.

Since $(V_t^X, V_t^Y)$ is a coupling of $P_t\delta_x$ and $P_t\delta_y$, we finally obtain
$$
\frac{\varepsilon}{2} \;  \le \; \P\left(  d(V^X_t ,V^Y_t) \le \delta  \right) \;  \le \; \sup_{\Gamma \, \in \, \mathscr{C}(P_t\delta_x, P_t\delta_y)} \Gamma \left( \left\{ (x', y') \in E \times E \colon  d(x',y') \le \delta \right\} \right).
$$
Therefore bound \eqref{LWI} holds and the semigroup $(P_t)$ is locally weak irreducible.
\end{proof}

The next proposition, established in \cite[Theorem 2.1]{HM11}, will help us to develop the uniqueness result based on local weak irreducibility and ASF+ with bounded $F$. \cite{HM11} stated the result in a Hilbert space structure, however, the proof can be conducted similarly by just replacing norms by the metric distance to our reference point $x_0$.

\begin{Proposition}[{\cite[Theorem 2.1]{HM11}}] \label{Lem} Let $(P_t)_{t \in \R_+}$ be a Markov semigroup satisfying ASF+. Let $\mu_1, \mu_2$ be two distinct ergodic invariant probability measures for $(P_t)_{t \in \R_+}$. Then, for any pair of points $(w_1, w_2) \in \operatorname{supp}(\mu_1) \times  \operatorname{supp}(\mu_2)$ one has
$$
d(w_1,w_2)\; \ge \; \frac{1}{F(d(w_1,x_0) \vee d(w_2,x_0))}.
$$
\end{Proposition}

Note that the above lemma tells us that not only are the supports of two different ergodic invariant measures under ASF+ disjoint (which is already the case under ASF -- see \cite[Theorem 3.16]{HM}) but they are even separated by a distance depending on the function $F$.
%

To present the proof of Theorem~\ref{Thm:2.4}, we need a couple of auxiliary statements.
\begin{Lemma} \label{Lem:1}
Let $(P_t)_{t\in \R_+}$ be a Feller semigroup and $x \in E$. Then for every $t >0$ and $A\subseteq E$ open such that $P_t(x,A)>0$, there exists  an open set $B:=B(x,t,A) \subseteq E$ containing $x$ such that
\begin{equation*}
\inf_{z \in B} \; P_t(z, A) \; > \; 0.
\end{equation*}
\end{Lemma}
\begin{proof}
Assume for the sake of a contradiction, that there exists $(z_n)_{n\in\N} \subseteq E$ such that $z_n \to x$ as $n\to \infty$ and $\lim_{n\to\infty}P_t(z_n,A)=0$.
Since $(P_t)_{t\in \R_+}$ is Feller, $P_t(z_n,\cdot)$ converges weakly to $P_t(x,\cdot)$. According to Portmanteau's theorem \cite[Theorem III.1.1.III]{Shi16} this implies
$$
0  \; = \; \limes P_t(z_n, A) \; \ge \;  P_t(x,A) \;  > \;  0,
$$
which yields a contradiction.
\end{proof}

\begin{Lemma}\label{Lem:2}
Let $\mu \in \cP(E)$ and $A \in \mathcal{E}$ such that $\mu(A)>0$. Then, $\overline{A}\, \cap \, \operatorname{supp}(\mu) \neq \emptyset$.
\end{Lemma}
\begin{proof}
Since $\mu(A)>0$ and the space $E$ is separable, we see that there exists $z_1 \in A$ such that
\mbox{$\mu(A  \cap B_{2^{-1}}(z_1)) >0$}. Now, let us inductively choose a sequence $(z_n)_{n\in\N} \subseteq A$ as follows. Due to separability of the space and $\mu(A \cap B_{2^{-n}}(z_n)) >0$, choose $z_{n+1} \in A \cap B_{2^{-n}}(z_n)$ such that $\mu(A \cap B_{2^{-n-1}}(z_{n+1})) >0$. Thus, for every $m\ge n$ we have
$$
z_m \in  B_{2^{-n}}(z_n)
$$
implying that $(z_n)_{n\in\N}$ is a Cauchy sequence. By the completeness of the space, there exists $z \in \overline{A}$ such that $z_n \to z$ as $n\to \infty$. Furthermore,  $z \in \operatorname{supp}(\mu)$, since for every $\varepsilon >0$ there exists $n \in \N$ such that $B_{2^{-n}}(z_n) \subseteq B_\varepsilon(z)$ and by monotonicity,
$$
\mu(B_\varepsilon(z)) \; \ge \; \mu(B_{2^{-n}}(z_n)) \; > \; \mu(A  \cap  B_{2^{-n}}(z_n))\;  > \;  0.
$$
Hence, $z \in \overline{A} \cap \operatorname{supp}(\mu)$.
\end{proof}

\begin{proof}[Proof of Theorem \ref{Thm:2.4}]
This proof is inspired by some ideas from \cite[proof of Corollary~1.4]{HM11}. Assume for the sake of a contradiction that there existed two distinct invariant probability measures $\mu_1, \mu_2 \in \cP(E)$. Since every invariant probability measure can be written as the convex combination of two ergodic measures (see, e.g.,  \cite[p. 670]{HM11}),  without loss of generality, we may assume $\mu_1, \mu_2$ to be ergodic.

Denote $M:=\|F\|_{\infty}<\infty$. For $i=1,2$, choose $u_i \in E$ such that $u_i \in \operatorname{supp}(\mu_i)$. Set $R:= d(u_1,x_0)  \vee  d(u_2,x_0)$, where $x_0$ is defined in \eqref{LWI}. Furthermore, let $\varepsilon>0$ be such that $M^{-1}>6\varepsilon$ and choose $T:=T(R, \varepsilon)$ as in (\ref{LWI}). Then, by the local weak irreducibility, there exists a coupling $\Gamma$ of $P_T\delta_{u_1}$, $P_T\delta_{u_2}$ with
$$
\Gamma \left( \Delta_\varepsilon \right) \;  > \;   0,
$$
where $\Delta_\varepsilon:=\{(x,y) \in E\times E \colon \;  d(x,y) \le \varepsilon\}$. Denote by $\Delta:=\{(x,x) \in E\times E \colon \; x \in E\}$ the diagonal on the product space.

Due to the fact that $E$ is separable, there exists a countable set $Q \subseteq \Delta$ which lies dense in $\Delta$. Hence,
$$
\union{q \in Q}{} B^{E\times E}_{2\varepsilon}(q) \;  \supseteq \; \Delta_\varepsilon,
$$
where $B^{E\times E}_\alpha((q_1,q_2)):=\{(x,y) \in E\times E\colon \; d(x,q_1)+d(y,q_2) < \alpha \}$ denotes the open ball in $E \times E$ of radius $\alpha$ around $q$ with respect to the sum of the distances in each coordinate. Therefore, there exists
$\tilde{v}_\varepsilon=(v_\varepsilon,v_\varepsilon) \in Q\subseteq \Delta$ such that
$\Gamma(B^{E \times E}_{2\varepsilon}(\tilde{v}_\varepsilon))>0$. Hence
$$
P_T(u_i,B_{2\varepsilon}(v_\varepsilon))  \; = \;   P_T\delta_{u_i}(B_{2\varepsilon}(v_\varepsilon)) \; \ge \; \Gamma(B^{E \times E}_{2\varepsilon}(\tilde{v}_\varepsilon)) \; > \; 0
$$
for every $i=1,2$. Thus, according to Lemma \ref{Lem:1}, there exists $\delta_\varepsilon>0$ such that
\begin{equation}\label{eq:3.6}
\inf_{z\,  \in\, B_{\delta_\varepsilon}(u_i)} \; P_T(z,B_{2\varepsilon}(v_\varepsilon)) \; > \;  0
\end{equation}
for every $i=1,2$. Additionally, since $u_i \in \operatorname{supp}(\mu_i)$, we have
\begin{equation} \label{eq:3.7}
\mu_i \left( B_{\delta_\varepsilon}(u_i) \right)  \; > \; 0
\end{equation}
for every $i=1,2$. Combining (\ref{eq:3.6}) and (\ref{eq:3.7}) and using the fact that $\mu_1, \mu_2$ are invariant for $(P_t)_{t\in \R_+}$, we obtain
\begin{align*}
\mu_i \left( B_{2 \varepsilon}(v_\varepsilon)\right)  \; = \;  \int_E  \, P_T(z, B_{2\varepsilon}(v_\varepsilon)) \;\mu_i(\diff z)  &\ge \;  \int_{B_{\delta_\varepsilon}(u_i)}  \, P_T(z, B_{2\varepsilon}(v_\varepsilon))\; \mu_i(\diff z)\\
& \ge\;  \mu_i\left( B_{\delta_\varepsilon}(u_i)\right) \;  \inf_{z \, \in \, B_{\delta_\varepsilon}(u_i)} \, P_T(z,B_{2\varepsilon}(v_\varepsilon))) \;  >  \; 0.
\end{align*}
Since, by monotonicity, $\mu_i \left( B_{3 \varepsilon}(v_\varepsilon)\right)>0$, Lemma~\ref{Lem:2} directly implies that
$$
B_{3\varepsilon}(v_\varepsilon) \, \cap \,  \operatorname{supp}(\mu_i)  \; \neq  \; \emptyset
$$
for $i=1,2$. This means that there exists $(w_1, w_2) \in \operatorname{supp}(\mu_1)\times \operatorname{supp}(\mu_2)$ such that $w_1, w_2 \in  B_{3\varepsilon}(v_\varepsilon)$. Thus, we finally get
$$
\frac{1}{M} \;  \le \; d(w_1 ,w_2) \; \le \; 6\varepsilon \; < \; \frac{1}{M},
$$
where the first inequality follows from Lemma \ref{Lem}, yielding the desired contradiction.
\end{proof}


\begin{thebibliography}{CGHV}

\bibitem[BKS]{BKS}
O. Butkovsky, A. Kulik and M. Scheutzow. Generalized couplings and ergodic rates for SPDEs and other Markov models. arXiv preprint arXiv:1806.00395 (2018).


\bibitem[CGHV]{CGHV}
P. Constantin, N. Glatt-Holtz and V. Vicol. Unique ergodicity for fractionally dissipated, stochastically forced 2D Euler equations. \textit{Communications in Mathematical Physics} 330.2 (2014): 819-857.


\bibitem[HM06]{HM}
M. Hairer and J. C. Mattingly. Ergodicity of the 2D Navier-Stokes equations with degenerate stochastic forcing. \textit{Annals of Mathematics} (2006): 993-1032.

\bibitem[HM08]{HM08}
M. Hairer and J. C. Mattingly. Spectral gaps in Wasserstein distances and the 2D stochastic Navier--Stokes equations. \textit{The Annals of Probability} 36.6 (2008): 2050-2091.

\bibitem[HM11]{HM11}
M. Hairer and J. C. Mattingly. A theory of hypoellipticity and unique ergodicity for semilinear stochastic PDEs. \textit{Electronic Journal of Probability} 16 (2011): 658-738.

\bibitem[KS91]{KS91}
I. Karatzas and S. E. Shreve. \textit{Brownian Motion and Stochastic Calculus}. 1991.

\bibitem[Par05]{Par05}
K. R. Parthasarathy. \textit{Probability measures on metric spaces}. 2005.

\bibitem[Ros87]{Ros87}
J. Rosen. Joint Continuity of the Intersection Local Times of Markov Processes. \textit{Annals of Probability} 15.2 (1987): 659-675.

\bibitem[Shi16]{Shi16}
A. N. Shiryaev. \textit{Probability-1}. 2016.

\bibitem[Tsy09]{Tsy09}
A. B. Tsybakov. \textit{Introduction to Nonparametric Estimation}. 2009.

\bibitem[V08]{V08}
C. Villani. \textit{Optimal transport: old and new}. Springer, 2008.


\end{thebibliography}
\end{document}